\documentclass[12pt]{amsart}

\usepackage{latexsym}
\usepackage{amssymb}
\usepackage{amsmath}
\usepackage{amsthm}

\usepackage{enumerate}

\usepackage[curve,arrow,matrix,tips]{xy}

\usepackage{amscd}

\usepackage{amsfonts}

\makeatletter
\@namedef{subjclassname@2010}{%
  \textup{2010} Mathematics Subject Classification}
\makeatother

\newtheorem{theorem}{Theorem}[section]
\newtheorem{lemma}[theorem]{Lemma}
\newtheorem{corollary}[theorem]{Corollary}

\theoremstyle{definition}
\newtheorem{definition}[theorem]{Definition}
\theoremstyle{remark}

\theoremstyle{question}
\newtheorem{question}[theorem]{Question}

\numberwithin{equation}{section}

\def\max{\mathop{\rm max}\nolimits}

\newcommand{\N}{{\mathbb N}}
\newcommand{\F}{{\mathcal F}}

\newcommand{\barr}[1]{\mathcal{#1}}

\newcommand{\sucb}[2]{(#1_s)_{s \in \mathcal{#2}}}

\newcommand{\matb}[2]{(#1 _s^{i})_{s \in #2, i \in \mathbb{N}}}
\newcommand{\matbr}[2]{(#1 _s^{i})_{s \in #2, i \in \mathbb{N}}}

\newcommand{\norma}[1]{|| #1 ||}
\newcommand{\summa}[3]{\sum_{i =1}^{#1} #2_i #3_i }
\newcommand{\summas}[4]{\sum_{i=1}^{#1} #2_i #3_{#4_i} }
\newcommand{\summaa}[4]{\sum_{i =1}^{#1} #2_i #3^i_{#4_i} }

\newcommand{\plegman}[2]{Plm_{#2}(#1)}

\newcommand{\EspN}[1]{\mathcal{N}_{#1}}
\newcommand{\FIN}{FIN}

\newcommand{\asyf}[1]{\Psi_{#1}}
\newcommand{\suc}[1]{(#1_i)_{i \in \mathbb{N}}}
\newcommand{\MA}[2]{AM_{#1}(#2)}
\newcommand{\SP}[2]{SM_{#1}(#2)}
\newcommand{\norm}[1]{||#1 ||}
\newcommand{\summaaa}[5]{\sum_{i=1}^{#1} #2_i #3^{{#5}_i}_{#4_i} }
\newcommand{\summasa}[4]{\sum_{i=1}^{#1} #2_i #3_{#4_i}}
\newcommand{\summaset}[3]{\sum_{i=1}^{#1} #2_i #3_i }
\newcommand{\summaaset}[4]{\sum_{i=1}^{#1} #2_i #3^i_{#4_i} }
\newcommand{\res}[2]{#1 \upharpoonright_{#2}}
\newcommand{\sobre}[2]{#1 / \{ #2\}}

\date{}

\title[Asymptotic models]{Asymptotic models via plegma families}

\author{E. A. Calderon-Garcia and S. Garcia-Ferreira}

\address{Centro de Ciencias  Matem\'aticas \\
Universidad Nacional Aut\'onoma de M\'exico
     Apartado Postal 61-3 \\
    Xangari, 58089 \\
      Morelia, Michoac\'an \\
    M\'exico}
\email{ngslivers@gmail.com, sgarcia@matmor.unam.mx}

\thanks{Research supported by CONACYT grant no. 176202 and PAPIIT grants no. IN-101911.}

\subjclass{Primary 22A10, 54H11.  Secondary 22B05.}

\keywords{asymptotic model, plegma family, barrier, spreading barrier}

\date{}

\begin{document}
\maketitle

\begin{abstract}
In the article \cite[Th. 68]{akt2}, it was shown that there exists a Banach space $X$ with Schauder basis $(e_i)_{i=1}^{\infty}$ which does not admit  $\ell_p$ as the model space obtained by a finite chain of sequences such that each element is a spreading model of a block subsequence of the previous element, starting from a block subsequence of $(e_i)_{i=1}^{\infty}$. We prove that $X$ has the stronger property of not admitting $\ell_p$ via a finite chain consisting of block asymptotic  models. This is related to a question posed by L. Halbeisen and  E. Odell in  \cite{asym} for the special case of block generated asymptotic models. Also, we show that for  every $k \in \N$ the Ramsey Coloring Theorem for $[\N]^{k}$  is equivalent to the following $k$-oscillation stability:

  In an arbitrary  Banach space $X$, for every $\epsilon > 0$ and for every normalized sequence $(e_i)_{i \in \N}$ in $X$ there exists  $M \in [\N]^{\infty}$ such that if $n_1,n_2,\cdots,n_k,m_1,m_2,\cdots,m_k \in M$, then
  $$
  \big| ||\sum_{i =1}^{k}a_i e_{n_i} || - ||\sum_{i=1}^{k} a_i e_{m_i} || \big| < \epsilon
  $$
   for all $(a_i)_{i=1}^{k} \in [-1,1]^k$.

\end{abstract}

\section{Preliminaries and Introduction}

In \cite{asym}, the authors used  Milliken's Theorem to introduce asymptotic models with a strong form of convergence with respect to a sequence of seminormalized sequences which satisfy some additional properties. On the other hand, in  the article \cite{akt}, a different type of convergence, related to the process of obtaining a spreading model of a block subsequence of another spreading model, is introduced and analyzed. In this paper, we  follow these ideas in the context of asymptotic models, by introducing higher order asymptotic models, and extend some of the known results from the higher order spreading models.
But, we shall use Ramsey Theorem rather than  Milliken's Theorem  to introduce a kind of higher order asymptotic model. In the next  two paragraphs, we shall include some notation and terminology  that we need throughout the paper.

\medskip

The symbol $FIN$ will denote the family of all finite subsets of $\mathbb{N}$. If $N$ is an infinite subset of $\mathbb{N}$, then we denote by $[N]^{\infty}$ the set of all infinite subsets of
$N$. To specify the elements of a $s \in FIN$  we shall write $s = \{s(1),\cdots, s(|s|)\}$ and in this notation  we shall always assume that $s(1) < s(2) <\cdots.< s(|s|)$. If $s, t \in FIN$, we say
that $s < t$ if $\max\{s\} < \min\{t\}$ (that is, $s(|s|) < t(1)$). In the case when $s = \{n\}$ we simply write $n < t$. If $s, t \in FIN$, then $s \sqsubseteq t$ means that $s$ is an initial
segment of $t$ and when $s < t$ we denote $s^{\frown}t = s \cup t$. If $s \in FIN$ and $A \in [\N]^\infty$, $s\sqsubseteq A$ also means that $s$ is an initial segment of $A$. If $A \in [\N]^\infty$ and $n \in \N$, then $A/n = \{ m \in A : n < m
\}$.  We denote for $\mathcal{F} \subset \FIN$, $\mathcal{F} / k = \{ s \in \mathcal{F} : \min(s) > k\}$. If $\F \subseteq FIN$ and $n \in \N$, then $\F_{\{n\}} = \{ s \in FIN : n < s \ \text{and} \ \{n\}^{\frown} s \in \mathcal{F}\}$. If $N
\in  [\N]^\infty$, then $\{n_1, n_2,\cdots.\}$ will stand for the increasing enumeration of $N$.

\medskip

Our Banach spaces will be separable, infinite dimensional and real. The sphere of a Banach space $X$ will be denoted by $S(X)$. The dual space of a Banach space will be denoted  by $X^*$. In case that several Banach spaces are involved and we want to specify the norm of a Banach space $X$ we shall write $\| \cdot \|_X$. We say that $(e_n)_{n \in \mathbb{N}}$ is a {\it Schauder basis} of $X$ if for each $x \in X$ there is a unique sequence of real numbers $(a_n)_{n \in \mathbb{N}}$ such that $ x = \sum_{n =1}^\infty a_ne_n$ and say that it is a {\it basic sequence} if it is a Schauder basis of the space they generate.  All the Banach spaces that we shall consider in this paper will have a Schauder basis. Thus, for our convenience,  a Banach spaces $X$ will be identified with a Schauder basis $(e_n)_{n \in \mathbb{N}}$ of it. If $(e_n)_{n \in \mathbb{N}}$ is a Schauder basis for $X$, then $e_i^*(\sum_{n = 1}^\infty a_ne_n) = a_i$ is a functional on $X$. Thus $\{(e_n,e_n^*)\}_{n \in \mathbb{N}}$ is a biorthogonal system. Let $X$ be a Banach space and let $(x_n)_{n \in \mathbb{N}}$ be a sequence in $X$.  The sequence $(x_n)_{n \in \mathbb{N}}$ is called {\it normalized} if $||x_n|| = 1$ for all $n\in \mathbb{N}$. We say that $(x_n)_{n \in \mathbb{N}}$ is $C$-basic if $\|\sum_{i = 1}^n a_ix_i\| \leq C\|\sum_{i = 1}^m a_ix_i\|$ for all $n < m$ and $(a_i)_{i = 1}^m \subseteq \mathbb{R}$. The minimum number $C$ with this property is called the {\it basic constant} of $(x_n)_{n \in \mathbb{N}}$ and will be denoted  by $bc(x_n)_{n \in \mathbb{N}} $.

\medskip

In the literature, one of the main problems regarding the notion of asymptotic structure  is to see if certain desirable spaces  can be obtained through the application of processes inside of a given Banach space. An important example about this is the celebrated Theorem of Krivine \cite{k} concerning the  finite representability of some $\ell_p$, with  $p \in (1, \infty]$, in any Banach space.
It is known (\cite{odell}) that there is a Banach space $X$ such that no $\ell_p$ can be generated as a  spreading model by any of its weakly null sequences.
In a recent paper by S. A. Argyros, V. Kanellopoulos and K. Tyros, \cite{akt2}, this last result was extended so that not even a finite chain of block generated spreading models starting from this space $X$ can generate an $\ell_p$. This question remained open for the case of asymptotic models introduced in \cite[P. 6.5]{asym}, that is:

\medskip

\begin{question}\label{6.5}
For any space X, does there exist a finite chain of asymptotic
models $X = X_0, X_1, \cdots,X_n$, so that $X_{i+1}$ is an asymptotic
model of $X_i$, for $i = 2,\cdots, n$, and $X_n$ is isomorphic to $c_0$ or $\ell_p$ for some
$1 \leq p < \infty $ ?
\end{question}

This question motives this paper. In fact,  we give a negative solution to Question \ref{6.5}  when the asymptotic
models are block generated or weakly generated.

\medskip

For the reader convenience and to explain the new concepts that we shall need, we recall several notions from the literature:

\medskip

$\bullet$ \ \ {\it Spreading model}: Let $(x_n)_{n \in \N}$ be a normalized basic sequence of a Banach space $X$. A Banach space $E$ with a Schauder basis $(e_n)_{n \in \mathbb{N}}$ is called a
{\it spreading model} of $(x_n)_{n \in \N}$ if there is  $\epsilon_n \searrow 0$ such that for every $s = \{s(1),\cdots., s(n)\} \in FIN$ with $s(1) \geq |s| = n$  we have that
$$
\big|\|\sum_{j = 1}^na_jx_{{s(j)}} \|_X -\|\sum_{j = 1}^na_je_j \|_E \big| < \epsilon_{s(1)}.
$$
for every  $(a_i)_{i =1}^ n \in [-1,1]^n$.
In this case, we say that $(x_n)_{n \in \mathbb{N}}$ generates  $(e_n)_{n \in \mathbb{N}}$ (or $E$) as a {\it spreading model}.

\medskip

A. Brunel and L. Sucheston \cite{bs} showed that every normalized basic sequence of a Banach space has a subsequence that generates a spreading model. The proof of this result
involves Ramsey Theorem. Indeed, in \cite{Shch}, the reader can find a beautiful proof of this fact using the theorem known as Ramsey's Theorem for Analysis (see
below Theorem \ref{ramseyanalistas}). We shall show, in the next section, that the Brunel-Sucheston Theorem is in fact equivalent to Ramsey's Theorem (see Theorem \ref{equiva}).

\medskip

The next definition of $K$-basic array from the paper \cite{asym} is used to generalize the notion of spreading model:

$\bullet$ \ \ {\it Basic matrix}:  A matrix $(x_m^n)_{n, m \in \mathbb{N}}$ of elements of $X$ is called a {\it basic matrix}
if $(x_m^n)_{m \in \mathbb{N}}$ is a normalized basic sequence in $X$ for each $n  \in \mathbb{N}$.

$\bullet$ \ \ {\it Asymptotic model}:   Let $(x_m^n)_{n, m \in \mathbb{N}}$  be a basic matrix of a Banach space $X$. We say that a Banach space $E$ with a normalized Schauder basis
 $(e_n)_{n \in \mathbb{N}}$ is an {\it asymptotic model} of $(x_m^n)_{n, m \in \mathbb{N}}$
 if for each $\epsilon_n \searrow 0$ there is a sequence $(k_i)_{i \in \mathbb{N}}$ in $\mathbb{N}$ such that for every $s = \{s(1),\cdots, s(n)\} \in FIN$ with $s(1) \geq |s| = n$   we have that
$$
\big|\|\sum_{j= 1}^n a_jx^j_{k_{s(j)}}\|_X -\|\sum_{j= 1}^n a_je_{j}\|_E \big| < \epsilon_{s(1)},
$$
for every $(a_i)_{i = 1}^n \in [-1,1]^{n}$.
When this happens we say that $(x_m^n)_{n, m \in \mathbb{N}}$ generates  $(e_n)_{n \in \mathbb{N}}$ (or $X$) as an {\it asymptotic model}.

It is known  in  that  if $(x_m^n)_{n, m \in \mathbb{N}}$ is a  basic matrix of a Banach space $X$ and $\epsilon_n \searrow 0$, then there is a sequence $(k_n)_{n \in \mathbb{N}}$ in
$\mathbb{N}$ such that
 for every $s, t \in FIN$ with $s(1) \leq |s| = n = |t| \geq t(1)$ we have that
$$
\big|\|\sum_{j = 1 }^n a_jx^j_{k_{s(j)}}\| -\|\sum_{j = 1}^n a_jx^j_{k_{t(j)}}\| \big| < \epsilon_{\min\{s(1),t(1)\}},
$$
for every $(a_i)_{i = 1}^n \in [-1,1]^{n}$.

\medskip

In the first section, we prove several results concerning uniform barriers that will be applied in further sections.  We  prove, in the second section, that Ramsey's Theorem  is equivalent to some kind of oscillation properties in Banach spaces. The third section is devoted to recalling the definition and some properties of plegma families which were introduced in \cite{akt}. Finally, in the last section we answer a problem related to Question \ref{6.5}.

\section{Uniform Barriers}
The Nash-William's Theory of Fronts and Barriers has been very important in the study of asymptotic models. Next we list some standard terminology of this theory that can be found in  \cite{ss}.

\medskip

Given an $\mathcal{F} \subseteq FIN$ and an infinite set $M \subseteq \mathbb{N}$,  $\mathcal{F} \upharpoonright_M$ denotes the set $ \{s \in \mathcal{F} : s \subseteq M \}$.

 \begin{definition} Let $\barr{B},\mathcal{C} \subseteq \FIN$.
\begin{itemize}
\item $\barr{B}$ is called {\it thin} if $s \not \sqsubseteq t$ for distinct $s, t \in \barr{B}$.

\item  $\barr{B}$ is called a {\it barrier} if:
 \begin{itemize}
 \item For every $M \subseteq \mathbb{N}$ there is an $s \in \barr{B}$ such that $s\sqsubseteq M$.
 \item For every $s,t \in \barr{B}$ if $t \neq s$ then $s \not \subseteq t$ and $t \not \subseteq s$.
 \end{itemize}
   $\barr{B}$ is a spreading barrier if, in addition
 \begin{itemize}
 \item If $s \in \barr{B}$ and $ r \in \FIN$ are such that $s(i) \leq r(i)$,  for all $i \in \mathbb{N}$, then there exists $t \in \barr{B}$ such that $r \sqsubseteq t$.
 \end{itemize}

 \item  We say that $\mathcal{B}$ has the  {\it Ramsey property} if for every partition $\mathcal{B} = \mathcal{P}_0 \cup\cdots\cup \mathcal{P}_k$ and for every $N \in [\mathbb{N}]^{\infty}$
     there is an  $M \in [N]^{\infty}$ such that at most one of the sets $\mathcal{P}_0\upharpoonright_M, \cdots, \mathcal{P}_k\upharpoonright_M$  is nonempty.
 \item We define $\mathcal{B} \oplus \mathcal{C} = \{ s \cup t : s \in \mathcal{B}, t \in \mathcal{C} \ \text{and} \ s<t \}$.
  \end{itemize}
 \end{definition}

  It is not hard to see if $\barr{B},\mathcal{C} \subseteq \FIN$ are both barriers, then $\mathcal{B} \oplus \mathcal{C}$ is also a barrier.

\medskip

   Given $\mathcal{D} \subseteq (FIN)^{< \infty}$ we say it has the Ramsey property if for every partition $\mathcal{D} = \mathcal{P}_0 \cup\cdots\cup \mathcal{P}_k$ and for every $N \in [\mathbb{N}]^{\infty}$
   there is an  $M \in [N]^{\infty}$ such that at most one of the sets $\mathcal{P}_0\upharpoonright_M, \cdots, \mathcal{P}_k\upharpoonright_M$  is nonempty. Where $\mathcal{P}_k\upharpoonright_M := \{d =(d_i)_{i=1}^{l} \in \mathcal{P}_k : \forall 1 \leq i \leq l(d_i \subseteq M) \}$. Notice that if the function $\phi : \mathcal{D} \to FIN$ that takes $d = (d_i)_{i=1}^{l}$ to $\cup_{i=1}^{l}d_i $ is injective, then this definition is equivalent to its image having the Ramsey property.

  \begin{definition} Let $\mathcal{F} \subseteq \FIN$ and $\alpha < \omega_1$. We say that $\F$ is $\alpha$-{\it uniform}
if:
 \begin{itemize}
 \item $\alpha = 0$ and $\F = \{\emptyset\}$.

\item $\alpha = \beta +1$ and $\F_{\{n\}}$ is $\beta$-uniform for each $n \in \N$.

\item $\alpha$ is a limit ordinal and there is a sequence $\alpha_n\nearrow \alpha$ such that  $\F_{\{n\}}$ is $\alpha_n$-uniform for each $n \in \N$.
\end{itemize}
We say that $\F$ is uniform if it is $\alpha$-uniform for some $\alpha < \omega_1$.
\end{definition}

It was shown by Nash-Williams (\cite[Lemma II.2.7]{ss}) that every thin family is Ramsey though for the special case of uniform barriers it can shown by a simple induction.

\medskip

Given $M = \{ m_k : k \in \N \}, N=\{ n_k : k \in \N \} \in [\N]^{\infty}$, we denote by $T_{M,N}$ the function  $m_k \to n_k$, and if $s =\{m_{k_1}, \cdots, m_{k_l} \} \in [M]^{<\infty}$, then we will write $T_{N,M}(s) = \{n_{k_1}, \cdots,  n_{k_l}\}$. Observe that given an uniform barrier $\mathcal{F}$ on $M$, $T_{M,N}(\mathcal{F})$ is a uniform barrier on $N$ with the same uniformity as $\mathcal{F}$. We omit the proof of the following easy lemma.

\begin{lemma} For $M,N \in [\mathbb{N}]^{\infty}$ and $k \in \N$, we have that
$$ T_{\sobre{M}{m_k}, \sobre{N}{n_k}} (s) = T_{M,N}(s), $$
 for each $s \in [M]^{< \infty}$ with $s > m_k$.
\end{lemma}

\begin{theorem}
\label{2.4}
Let $\barr{F}, \barr{G} \subseteq FIN$ be uniform barriers such that $\barr{G}$ is spreading and $o(\barr{F}) \leq o(\barr{G})$. Then for every $M,N \in [\mathbb{N}]^{\infty}$ there exists an $L_0 \in [N]^{\infty}$ such that
$$\forall L \in [L_0]^{\infty}  (T_{M,L}(\barr{F}\upharpoonright_M) \subseteq \overline{\barr{G}\upharpoonright_L}^{\sqsubseteq} ).$$
\end{theorem}
\begin{proof}

The proof will be by induction on $o(\barr{G})$. The case when $o(\barr{G}) < \omega$ is straightforward. Now, let $\alpha < \omega_1$ and assume that the result holds for each barrier with uniformity $< \alpha$. Let $\barr{F},\barr{G} \subseteq FIN$ be uniform barriers with $o(\barr{F}) \leq o(\barr{G})= \alpha$. Fix $M,N \in [\mathbb{N}]^{\infty}$. Recursively, for each $k \in \N$ we shall define $L'_k \in [N]^\infty$ and $l_k \in \N$ so that
\begin{enumerate}
\item $L'_0 = N$ and $l_0 = 0$.
\item $l_{k+1} \in L'_{k}$ for every $k \in \N$.
\item For every $k \in \N$, $L'_{k+1} \in [L'_{k} / \{l_{k+1}\}]^{\infty}$ and
$$
\forall L \in [L'_{k+1}]^{\infty}  (T_{M/ \{m_k\},L}(\barr{F}_{\{m_k\}}\upharpoonright_{M/ \{m_k\}}) \subseteq \overline{\barr{G}_{\{l_k\}}\upharpoonright_L}^{\sqsubseteq} ).
$$
\end{enumerate}
To carry out the construction take $n \in \N$ and assume that $L_i$ and $l_i$ have been defined for each $i \leq n$. By definition of uniformity there must exist a $l_{n+1} \in L'_n$ such that $o(\barr{F}_{\{ m_{n+1}\}}) \leq o(\barr{G}_{\{ l_{n+1}\}}) < \alpha$. Then,  we apply the induction hypothesis to these barriers and $\sobre{M}{m_{n+1}}$ and $\sobre{L'_n}{l_{n+1}}$ to obtain $L'_{n+1} \in [N]^{\infty}$ such that the conclusion holds. Let us prove that the set $L_0 = \{l_k  : k \in \N \setminus \{0\} \}$ satisfies the conclusion of the theorem. Indeed, to see this take any $L \in [L_0]^\infty$ and notice that  $L / \{ l_{k}\} \subseteq L'_k$, for all $k \in \N$. Now take $s = \{ m_{k_1}, \cdots , m_{k_s} \} \in \barr{F}\upharpoonright_M$. Set $s' = s / \{ m_{k_1}\} \in \barr{F}_{{\{m_{k_1}\}}}$ and denote the $k_1$ element of $L$ by $l$. It is clear that $l \geq l_{k_1}$. By the previous lemma and  clause $(3)$, we have that
$$
T_{M,L}(s') = T_{\sobre{M}{m_{k_1}}, \sobre{L}{l}}(s') \in \overline{\barr{G}_{\{l_{k_1}\}}}^{\sqsubseteq}.
$$
  So, $ \{ l_{k_1}\} \cup T_{M,L}(s')  \in \overline{\barr{G}}^{\sqsubseteq}$ and since $\barr{G}$ is spreading and $l \geq l_{k_1}$ it follows that
  $$
  T_{M,L}(s) = \{ l\} \cup T_{M,L}(s') \in \overline{\barr{G}}^{\sqsubseteq}.
  $$
\end{proof}

\begin{corollary}\label{2.5}
Let $\barr{F}, \barr{G} \subseteq FIN$ be uniform barriers such that $\barr{G}$ is spreading and $o(\barr{F}) \leq o(\barr{G})$.
Let $M,N \in [\mathbb{N}]^{\infty}$ and assume that $L_0 \in [N]^{\infty}$ that satisfies the conclusion of Theorem \ref{2.4}. Then, for every $L \in [L_0]^{\infty}$
and every  $s \in \barr{G}\upharpoonright_{L}$ there is a unique element, denoted by $\psi_{L,M}(s)$, of $\barr{F}\upharpoonright_M$ satisfying $T_{M,L}(\psi_{L,M}(s)) \sqsubseteq s$.
\end{corollary}

\begin{proof}
Take $s \in \barr{G}\upharpoonright_{L}$. Since $T_{M,L}(\barr{F}\upharpoonright_M)$ is a barrier on $L_0$ there exists an $t \in \barr{F}\upharpoonright_M$ such that either $T_{M,L}(t) \sqsubseteq  s$ or $s \sqsubset T_{M,L}(t)$. Suppose that the second option holds, then by the conclusion of Theorem \ref{2.4} it follows that there exists $s' \in \barr{G}\upharpoonright_{L}$ such that $s \sqsubset T_{M,L}(t) \sqsubseteq s'$ which is impossible since $\barr{G}\upharpoonright_{L}$ is a barrier. Therefore, $T_{M,L}(t) \sqsubseteq  s$.
\end{proof}

If follow from Corollary \ref{2.5} that
$$
\psi_{L,M} : \barr{G}\upharpoonright_{L} \rightarrow \barr{F}\upharpoonright_M
$$
is a well-defined function. It follows directly from Theorem \ref{2.4} that this function is also surjective.

\section{Some Remarks on Ramsey's Theorem}

First, we recall a nice construction of Banach spaces. Given an infinite set $I$,  $c_{00}(I)$ will denote the set of all  finitely supported functions from $I$ to $\mathbb{R}$. For $x \in
c_{00}(I)$, we let $supp(x) = \{ i \in I : x(i) \neq 0 \}$  stand for the support of $x$.
Given $i \in I$, we let $e_i: I \to \mathbb{R}$ the function defined by $e_i(i) = 1$ and $e_i(j) = 0$ if $i \neq j$. Let $G_0 = \{\pm e_i^{*} : i \in I \}$. To define a seminorm on  $c_{00}(I)$
we can choose $W \subseteq c_{00}(I)^*$, we shall refer to this as the {\it norming set}, and define
$$
\|x\|_W = \sup\{ |f(x)| : f \in W\}
$$
 for every $x \in c_{00}(I)$. It is clear that $\| \cdot \|_W$ is a seminorm on $c_{00}(I)$. In particular, if $G_0
 \subseteq W$ then $\| \cdot \|_W$ is a norm . The completions of the norm spaces  of the form  $(c_{00}(I),\| \cdot \|_W)$ are sources of many interesting Banach spaces. Notice that if $I$ is
 finite its completion is $\mathbb{R}^{I}$.

\medskip

Now we state the The Ramsey Theorem for Analysts:

\begin{theorem}{\bf [Ramsey Theorem for Analysis]}\label{ramseyanalistas} Let $(X,d)$ be a compact metric space and  $\barr{F} \subseteq FIN$ a family with the Ramsey property. For every function $F:\barr{B} \to X$ and for every sequence
$\epsilon_j \searrow 0$ there are $M =\{ m_1, m_2,\cdots\} \in [\N]^\infty$ and $x \in X$ such that
$$
d(F(s),x) < \epsilon_{min(s)},
$$
for each $s \in \res{\barr{F}}{M}$.
\end{theorem}

Based on the previous theorem we say that $F$ {\it converges} to $x$ on $M$ if there exists $\epsilon_j \searrow 0$ such that $d(F(s),x) < \epsilon_{min(s)}$ for each $s \in \res{\barr{F}}{M}$. Let us remark that if $N \in [\mathbb{N}]^{\infty}$ satisfies that $N /m \subseteq M$ for some $m \in \mathbb{N}$, then it is also true that $F$ converges to $x$ on $N$.

\medskip

As far as  we know Ramsey's Theorem for Analysis was stated in this form for the first time in \cite{Shch}. Let us  make some important comments about the proof of Brunel-Sucheston Theorem, which
lies in \cite{Shch}, by using Ramsey's Theorem for Analysis. Since our objective is to apply this theorem to obtain different asymptotic structures it is tempting to use (as the $(X,d)$ in (\ref{ramseyanalistas}), for each $k \in \N$, the metric space
$$
\mathcal{M}_k = \{ \|\cdot\|: \mathbb{R}^k \to [0,\infty) : \|\cdot\| \ \text{is a norm and} \ \forall i= 1,\cdots, k(\|e_i\| = 1) \}
$$
with metric given by
$$
d_n(\|\cdot\|_1,\|\cdot\|_2) = \sup \{ \big|\| a \|_1 - \|a\|_2 \big| : a \in [-1,1]^k \}.
$$
for every couple of norms $(\|\cdot\|_1,\|\cdot\|_2)$ of $\mathcal{M}_k$, but these spaces are not compact. In fact, a simple example to see this, for $k=2$, is the sequence of norms on $\mathbb{R}^2$ defined by the sets $W_n = \{ e^{*}_1- e^{*}_2, \frac{1}{n} e^{*}_2, \} \subseteq c_{00}(2)$. It is easy to see that every one of this sets defines a norm element of $\mathcal{M}_2$ and such that the sequence is Cauchy and yet it cannot converge to a norm (since the vector $(1,1)$ would necessarily have ``norm'' $0$ ). If we replace ``norm''  by ``seminorm'' as follows
$$
\mathcal{N}_k = \{ \rho : \mathbb{R}^k \to [0,\infty) : \rho \ \text{is a seminorm and} \ \forall i= 1,\cdots, k(\rho (e_i) = 1) \},
$$
where $\{e_i : 1 \leq i \leq n \}$ be the canonical base of $\mathbb{R}^{n}$, then the assertion is true. We remark that for every $\rho  \in \mathcal{N}_k$ we have that $|\rho(x)| \leq \|x \|_{\ell_1}$ for each $x \in \mathbb{R}^{k}$. Theorem \ref{ramseyanalistas} is often either applied implicitly or its proof is explicitly repeated to various closed subsets of $\mathcal{N}_k$ to obtain certain asymptotic structures.
To guarantee that the seminorms obtained by applying the Ramsey Theorem for Analysts we shall use the following closed subsets of $\mathcal{N}_k$:

$$
  \{ || \cdot || : \mathbb{R}^k \to [0,\infty) : || \cdot ||\ \text{is a norm and} \ \forall i= 1,\cdots, k(||e_i|| = 1)
  $$
  $$
  \text{and} \ bc(e_i)_{i =1}^{k} \leq 2 \}
 $$ and
$$
\{ || \cdot || : \mathbb{R}^k \to [0,\infty) : || \cdot ||\ \text{is a norm and} \ \forall i= 1,\cdots, k(||e_i|| = 1)
$$
$$
\text{and $(e_i)_{i=1}^{k}$ is $2$ unconditional}  \}.
$$

\medskip

Next, we shall state and prove that the Ramsey Coloring Theorem is equivalent to the following notion of oscillation stability on Banach spaces.

\begin{definition}\label{oci} Let $k \in \N$ and let $X$ be a Banach space. We say that a normalized sequence $(e_i)_{i \in \N}$  in $X$  is {\it $k$-oscillation stable} if for every $\epsilon>0$ there exists $M \in \N$ such that for each $n_1,n_2,\cdots,n_k,m_1,m_2,\cdots, m_k \in M$, we have that
$$
\big| ||\sum_{i =1}^{k}a_i e_{n_i} || - ||\sum_{i=1}^{k} a_i e_{m_i} || \big| < \epsilon,
$$
 for all $(a_i)_{i=1}^{k} \in [-1,1]^{k}$.
\end{definition}

Let us explain the name of the notion introduced in  Definition \ref{oci}. In  \cite[Def. III.5.4]{ss},  a function $f : S(X) \to  \mathbb{R}$ called  {\it oscillation
stable} on $X$ if for all infinite dimensional closed subspaces $Y$ of $X$ and $\epsilon >0$ there
is a closed infinite dimensional subspace $Z$ of $Y$ such that
$$
\sup{ \{|f(x) - f(y)| : x, y \in  S(Z) \} } < \epsilon.
$$
Now, given a normalized sequence $(e_i)_{i \in \N}$, it is possible to define the function $\Phi_{k} : [N]^{k}  \rightarrow \mathcal{N}_k$ that takes $F \in \mathcal{N}_k$ to the seminorm generated by $<e_i>_{i \in F}$, since $\mathcal{N}_k$ is a compact metric, by analogy, we get the concept of $k$-oscillation stability  for a normalized sequence.

\begin{theorem}\label{equiva} For every $k \in \N$, the following statements are equivalent.
\begin{enumerate}

\item For every function $C: [\N]^{k} \rightarrow 2$ there exists $i < 2$ and $M \in [\N]^{\infty}$ such that $[M]^{k} \subseteq C^{-1}(i)$.

\item Analysts Ramsey Theorem for $[\N]^{k}$.

\item  In an arbitrary  Banach space $X$, for every $\epsilon > 0$ and for every normalized sequence $(e_i)_{i \in \N}$ in $X$ there exists  $M \in [\N]^{\infty}$ such that $(e_i)_{i \in
    M}$ is  $k$-oscillation stable.
\end{enumerate}
\end{theorem}

\begin{proof} We only need to show the implication $(3) \Rightarrow (1)$.  Let $C : [\mathbb{N}]^{k} \rightarrow 2$ be a function. Then, we define the following norming set
  $$
  W = G_0 \cup \{  \sum_{i \in s} \pm e^{*}_i : s \in [\mathbb{N}]^{k}(C(s) =0 )\}.
  $$
  Notice  $(e_i)_{i \in \mathbb{N}}$ is a normalized sequence. We claim that if $C(t) =1$, then $ ||\sum_{j \in t} e_j||_{W} \leq k-1$. Indeed, suppose that $C(t) =1$ and fix $f \in W$. If
  $f \in G_0$, it is then evident that $f(\sum_{j \in t} e_j) \leq 1$. If $f \in W \setminus G_0$, then there is a $u \in [\mathbb{N}]^{k}$ with $C(u) = 0$ such that $supp(f) = u$. This means
  that $u \not = t$. As $|t| = |u|$, there is  $j_0 \in t \setminus u$. So,
  $$
  f(\sum_{j \in t} e_j) = f(\sum_{j \in t, j \not = j_0} e_j) \leq k-1.
  $$
   It then follows from the definition of the norm $|| \cdot ||_{W}$ that $||\sum_{j \in t} e_j||_{W} \leq k-1$.
   By applying the hypothesis to $\epsilon = \frac{1}{2}$, there exists a $k$-oscillation stable subsequence $(e_{n_i})_{i \in \mathbb{N}}$.  Set $M = \{n_i : i \in \mathbb{N} \}$. Then, we have
   that
  $$
  \left| ||\sum_{i \in s} e_i||_{W}  - ||\sum_{j \in t} e_j||_{W}   \right| \leq \frac{1}{2},
  $$
for every $s, t \in [M]^{k}$. We are done if $C(s) = 1$ for every $s \in [M]^{k}$.
  Suppose that this is not the case. Then, choose $s \in [M]^{k}$ with $C(s) =0$. Notice that
 $||\sum_{i \in s} e_i||_{W} = |s| = k$. Hence, if $t \in [M]^{k}$, then
  \begin{align*}
  \left| ||\sum_{i \in s} e_i||_{W} - ||\sum_{j \in t} e_j||_{W}  \right| &\leq \frac{1}{2} \\
  k - ||\sum_{j \in t} e_j||_{W} &\leq \frac{1}{2}\\
  k - \frac{1}{2} &\leq ||\sum_{j \in t} e_j||_{W}.
  \end{align*}
  So, by the above claim, we obtain that  $C(t) = 0$ for every $t \in [M]^{k}$.
\end{proof}

\begin{corollary}\label{equiva2} The following statements are equivalent.
\begin{enumerate}

\item Ramsey Theorem.

\item Analysis Ramsey Theorem.

\item  Brunel-Sucheston Theorem.
\end{enumerate}
\end{corollary}

\section{Plegma Families}

We start this section with a small modification of the notion of plegma family introduced in \cite{akt}.

 \begin{definition}\label{plegma}
 Let $n \in \N$. A finite sequence  $(s_i)_{i=1}^n$  of $FIN$ is called {\it plegma} if the following conditions hold:
 \begin{enumerate}
 \item $|s_1| \leq |s_2| \leq\cdots\leq |s_n|$, and

 \item $s_i(k) < s_{i+1}(k) <\cdots.< s_n(k) < s_j(k+1)$ for each $i, j = 1,\cdots, n$ and for each $0 < k \leq \min\{|s_i|,|s_j|\}$.
 \end{enumerate}
For  $\barr{B} \subseteq \FIN$  and   $n \in \N$, the $n$-{\it plegma family} of $\barr{B}$ is the family
 \begin{align*}
 \plegman{\barr{B}}{n} = \{ (s_i)_{i=1}^n \in \barr{B}^{n} :   (s_i)_{i=1}^n \ \textit{is plegma} \ \}.
 \end{align*}
 \end{definition}

    We remark that this definition of plegma family allows the finite sets be empty. For instance the family $(\emptyset, \emptyset, \cdots, \emptyset)$ is always a plegma family according to our definition. This is very useful in the induction steps of several proofs in this paper.

  \medskip

  In the article \cite{akt},
 the authors showed that if $\barr{B}$ is a spreading barrier, then  $\plegman{\barr{B}}{n}$  has  the Ramsey property for all $n \in \mathbb{N}$.

\medskip

  In the next lemma we prove that for finitely many uniform barriers of increasing uniformity we can always find plegma sequences with very strong properties somehow related to each given
  barrier.

 \begin{lemma}
 \label{util}
 Let $n \in \N$ and $\barr{B}_1, \barr{B}_2,\cdots, \barr{B}_n$ uniform barriers on $\mathbb{N}$ of uniformities $\gamma_1, \gamma_2, \cdots, \gamma_n$, respectively, such that $0 < \gamma_1 \leq \gamma_2 \leq \cdots \leq \gamma_n$. Then for every $M \in [\mathbb{N}]^{\infty}$ there are $m_1 < m_2 \cdots< m_n \in M$ such that $\barr{B}_1{_{\{m_1\}}},\barr{B}_2{_{\{m_2\}}},\cdots,\barr{B}_n{_{\{m_n\}}}$ have uniformity  $\lambda_1 \leq \lambda_2\cdots\leq \lambda_n $, respectively, and $\lambda_n < \gamma_n$.
 \end{lemma}
 \begin{proof}
   Take an arbitrary $m_1 \in \mathbb{N}$ and notice that $\barr{B}_1{_{\{m_1\}}}$ has uniformity $\lambda_1  < \gamma_1$. Since $\gamma_1 \leq \gamma_2$, we claim that there is $m_2 \in M$ with $m_2 > m_1$ such that $\barr{B}_2{_{\{m_2\}}}$  has uniformity $\lambda_2$ and $\lambda_1 \leq \lambda_2 < \gamma_2$. Indeed, in the case that $\gamma_2$ is a limit ordinal, we have by definition of uniformity that $sup\{ unif(\barr{B}_2{_{\{j\}}} ) : j \in \mathbb{N}  \} = \gamma_2$. In the case that $\gamma_2$ is a successor ordinal, again by definition of uniformity, it follows that $unif(\barr{B}_2{_{\{j\}}}) = \gamma_2 -1$ for each $j \in \mathbb{N}$. In both cases the choice of an appropriate $m_2$ is simple.  By applying this argument recursively, we define natural numbers $m_1 <m_2 <\cdots< m_n$ in $M$ and  barriers $\barr{B}_1{_{\{m_1\}}},\barr{B}_2{_{\{m_2\}}},\cdots,\barr{B}_n{_{\{m_n\}}}$ of uniformity  $\lambda_1 \leq \lambda_2 \cdots \leq \lambda_n $, respectively, such that $\lambda_n < \gamma_n$.
 \end{proof}

 \begin{lemma}
 \label{existenciaplegma}
 Let $n \in \N$ and $\barr{B}_1, \barr{B}_2,\cdots, \barr{B}_n$ uniform barriers on $\mathbb{N}$ of uniformities $\gamma_1, \gamma_2, \cdots, \gamma_n$, respectively, satisfying $\gamma_1 \leq \gamma_2 \leq \cdots \leq \gamma_n$. Then for every  $M \in [\mathbb{N}]^\infty$ there exist $s_1 \in \barr{B}_1\upharpoonright_M, s_2 \in \barr{B}_2\upharpoonright_ M,\cdots,s_n \in \barr{B}_n\upharpoonright_ M$ such that the sequence $(s_i)_{i =1}^{n}$ is  plegma.
 \end{lemma}

 \begin{proof}
 The proof goes by induction on $\gamma_n$. Assume that  $\gamma_n < \omega$. We know that  $[M]^{\gamma_j}$ is the unique $\gamma_j$ uniform barrier on $M = \{m_i : i \in \N\}$ for each $1 \leq
 j \leq n$. So, for every $1\leq j \leq n$, we take  $s_j \in \barr{B}_j$ so that
 \begin{align*}
 s_j \sqsubseteq M_j := \{m_i : i \equiv j \ \text{mod} (n) \}.
 \end{align*}
  Notice that $|s_j| = \gamma_j$ for each $1\leq j \leq n$. So, by hypothesis, we get $|s_1| \leq |s_2| \leq\cdots\leq |s_n|$.  It is not hard to see that $(s_j)_{j=1}^n$ has the desire properties. Next, we assume  that the theorem holds for uniform barriers of uniformities $\lambda_1 \leq \lambda_2,\cdots,\leq \lambda_n $, respectively, where $\lambda_n < \gamma_n$. By applying Lemma \ref{util} to $M$ and the barriers $\barr{B}_{i_0}, \barr{B}_{i_0 +1},\cdots, \barr{B}_n$, where $i_0 \in \mathbb{N}$ is the first element satisfying $\gamma_{i_0} > 0$ we can find $m_{i_0} < m_{i_0 +1} \cdots< m_n \in M$ such that $\barr{B}_{i_0}{_{\{m_{i_0}\}}},\barr{B}_{i_0 +1}{_{\{m_{i_0 +1}\}}},\cdots,\barr{B}_n{_{\{m_n\}}}$ have uniformity  $\lambda_{i_0} \leq \lambda_{i_{i_0 +1}}\cdots\leq \lambda_n $, respectively, and $\lambda_n < \gamma_n$. Now, we  apply the induction hypothesis to these barriers and  $M' = M/\{ m_n\}$ to find $s'_i \in \barr{B}_i{_{\{m_i\}}} \upharpoonright_ {M'}$, for each $i_0 \leq i \leq n$, satisfying properties $(1)$ and $(2)$ of Definition \ref{plegma}. It is easy to see
  that the family $s_j = \emptyset$ for $j < i_0$ and $s_{i_0} = \{m_{i_0}\} \cup s'_{i_0}, s_{i_0 +1} = \{m_{i_0 +1}\} \cup s'_{i_0 +1},\cdots,s_n = \{m_n\} \cup s'_n$ is plegma.
 \end{proof}

By using previous theorems, we apply Theorem  \ref{ramseyanalistas} as follows:

 \begin{theorem}
 \label{ramseyanalistas2}
 Let $(X,d)$ be a compact metric space, $\barr{B}$ an spreading barrier and $n \in \N$. For every function $F: \plegman{\barr{B}}{n} \rightarrow X$ and for every sequence $\epsilon_i\searrow 0$
 there exists  $M= \{ m_i : i \in \mathbb{N}\} \in [\mathbb{N}]^\infty$ and $x \in X$ such that
 \begin{align}
 \label{nice}
 \forall l \in \mathbb{N}\forall s \in \plegman{\barr{B}\upharpoonright_{M/{m_l}}}{n}(d(F(s),x) < \epsilon_{l})
 \end{align}
 \end{theorem}

 \begin{proof}
 It is easy to see that the function $\phi$ defined\footnote{See definition of $\phi$ on page $3$ after Definition 2.1.} on $\plegman{\barr{B}}{n}$ is injective and that the image of $\barr{B}$ is a thin family. In particular, it has the Ramsey property. The conclusion follows by applying Theorem \ref{ramseyanalistas} to this image.
 \end{proof}

This theorem will be apply, in the next section,   to the metric space $(\EspN{k},d_k)$, for  $k \in \N$, but for the sake of completeness  let us show that it is  compact.

  \begin{lemma}
  For every natural number $k$, the space $(\EspN{k},d_k)$ is compact and metric.
  \end{lemma}

  \begin{proof}
  We prove that the space is both complete and totally bounded. Indeed, it is not hard to se  that it is complete since any Cauchy sequence of $(\EspN{k},d_k)$ determines a seminorm on $\mathbb{R}^k$. To see that it is totally bounded fix $\epsilon >0$ and choose   a finite $\frac{\epsilon}{4}$-net $B$ in $([-1,1]^{k},\norm{\cdot}_{\ell_1})$ and $A$  a  finite $\frac{\epsilon}{4}$-net in $[0,k]$. For every function $f : B \rightarrow A$ define
  \begin{align*}
  C_f = \{\rho \in \EspN{k}: \forall b \in B (|\rho(b) - f(b)| < \frac{\epsilon}{4} ) \}.
  \end{align*}
   Notice that any two elements $\rho_1, \rho_2 \in C_f$ satisfy that $d_k(\rho_1, \rho_2) < \epsilon$. To see this take $x \in [-1,1]^{n}$ and $b \in B$ so that $\norm{x-b}_{\ell_1} < \frac{\epsilon}{4}$.
   Then, we have that
   \begin{align*}
   |\rho_1(x) - \rho_2(x)| &\leq |\rho_1(x) - \rho_1(b)| + |\rho_1(b) - f(b)| \\
   &+ |f(b) - \rho_2(b)| + |\rho_2(b) - \rho_2(x) |  \\
   &\leq 2 \norm{x-b}_1 + 2 (\frac{\epsilon}{4}) < \epsilon.
   \end{align*}
      Now for every nonempty $C_f$ we fix  $\rho_f \in C_f$. We claim that  $\{\rho_f : f : B \rightarrow A \}$ is an
   $\epsilon$-net in $\EspN{n}$.  Indeed, for $\rho$ in $\EspN{n}$ we consider the function $f: B \to A$ defined by
  \begin{align*}
  f(b) = \min\{a \in A : |\rho(b) -a| < \frac{\epsilon}{4} \},
  \end{align*}
  this set is not empty since $0 \leq \rho(b) \leq \norm{b}_{\ell_1} \leq k$ and by the definition of $A$. It is clear that  $\rho$ is an element of $C_{f}$ which implies, by our claim, that $d_n(\rho_f, \rho) < \epsilon$.
   \end{proof}

 \section{$\mathcal{F} \times \N$- matrices}

 The purpose of this section is to apply Theorem~\ref{ramseyanalistas2} to obtain higher order asymptotic models. We follow the basic idea of asymptotic models from \cite{asym} and the
extension of spreading model from \cite{akt}.

 \begin{definition} Let $X$ be a Banach space and let $\mathcal{F} \subseteq \FIN$.  An $\mathcal{F}$-{\it sequence}  in $X$ is a sequence $\sucb{x}{\barr{F}}$ in $S(X)$
 indexed by elements of $\mathcal{F}$. A sequence of $\mathcal{F}$-sequences $\matb{x}{F}$   will be named $\mathcal{F} \times \N$-{\it matrix}.
 \end{definition}

 Following the analog definition in \cite{akt} for spreading models, we present the concept of a higher order asymptotic model.
 \begin{definition}
 Let $X$ be a Banach space, $\barr{B}$ a barrier and $\matb{x}{\barr{B}}$ a $\mathcal{B} \times \mathbb{N}$-matrix on $X$. We say that  a Banach space  $E$ with a normalized Schauder basis $(e_i)_{i \in \mathbb{N}}$ is an {\it asymptotic model} of  $\matb{x}{\barr{B}}$ if there exists $\epsilon_m \searrow 0$  such that for every $n \in \N$
 $$
 \forall m \in \mathbb{N}/ n \forall s \in (\plegman{\barr{B}/m}{n}) \Big( \Big| ||\sum_{i =1}^{n}{a_i e_i}||_E - ||\sum_{i=1}^{n}{a_i x^{i}_{s_i}} ||_X \Big| < \epsilon_m \Big),
$$
 for every $(a_i)_{i = 1}^n \in [-1,1]^{n}$. If  $\barr{B}$ is a  $\xi$-uniform barrier, then we say that the  $\mathcal{B} \times \mathbb{N}$-matrix $\matb{x}{\barr{B}}$ generates $(e_i)_{i \in \mathbb{N}}$ as an {\it asymptotic model of order} $\xi$.
 \end{definition}
 Notice that the usual definition of asymptotic model is easily recovered by using the barrier $\mathbb{N}$ in the previous definition. Given a Banach space $X$, we denote by $\MA{\xi}{X}$ the set of all asymptotic models of order $\xi$ and will refer to the elements of $\MA{\xi}{X}$ as the $\xi$-asymptotic models of $X$.

\medskip

Generalizing the notion of subsequence of a sequence in the context of $\barr{B}\times \N$-matrices we have the following.

 \begin{definition}
 Let $\matb{x}{\barr{B}}$ be an $\barr{B}\times \N$ -matrix. A {\it submatrix} of  a  $\barr{B}\times \N$-matrix $\matb{x}{\barr{B}}$ is a matrix of the form  $\matbr{x}{\barr{B} \upharpoonright_ M}$ where $M \in [\mathbb{N}]^\infty$.
 \end{definition}

Below, we shall see that it is possible to replace a submatrix indexed on a uniform barrier that generates certain higher order asymptotic model by one indexed on a spreading barrier with the same uniformity and generates the same higher order asymptotic model. In order to do this, we first establish some preliminary results.

\begin{lemma}
\label{plegmagrande}
Let $n \in \N $. If $L \in [\N]^{\infty}$ and $M \in [L]^{\infty}$ satisfy that
$$
|(m, m') \cap L| \geq n,
$$
for each $m, m'$ in $M $ with $m < m'$, then for every plegma $(t_i)_{i =1}^{n}$ in $M$ there exists a plegma $(t'_i)_{i=1}^{n} \in [[L]^{<\infty}]^{n}$ such that
\begin{itemize}
\item $t'_i \subseteq L / \{ max \{j \in M : j < min\bigcup_{i =1}^{n} t_i \}\}$, for each $i=1, \cdots, n$.
\item  $t_i \sqsubseteq t'_i$, for each $i=1, \cdots, n$.
\item $|t'_1|=|t'_2|= \cdots=|t'_n| = |t_n|$.
\end{itemize}
\end{lemma}

\begin{proof}
The proof will be by induction on $|t_n|$. In the case where $|t_n| = 1$ take $A = \{1 \leq  i \leq n : t_i = \emptyset \}$ and $B = \{1 \leq i \leq n : t_i \neq \emptyset \}$. For each $i \in A$ take a $l_i \in L \cap ( max \{j \in M : j < min\bigcup_{i =1}^{n} t_i \},min\bigcup_{i =1}^{n} t_i ) $ such that $i < j$ implies $l_i < l_j$. This can be done because of the hypothesis on the size of this intersection. It is clear that $t'_i = \{l_i \}$ for each $i \in A$ and $t'_i = t_i$ for each $i \in B$ satisfies the conclusion.
 Now suppose the result holds for every plegma $(t_i)_{i =1}^{n}$ where $|t_n| = k$ and take a plegma $(r_i)_{i =1}^{n}$ with $|r_n| = k+1$. Now take the plegma family $(s_i)_{i =1}^{n}$ defined as
 $s_i = \emptyset$ if  $r_i = \emptyset$ and $s_i = r_i \setminus \{ min(r_i)\}$ if $r_i \neq \emptyset $.  Since $|s_n| = k$ we can apply the induction hypothesis to get $(s'_i)_{i=1}^{n}$ that satisfies the conclusion of the lemma. It is not hard to see that $(r'_i)_{i=1}^{n}$, where $r'_i = s'_i $ if $r_i = \emptyset$ and $r'_i = s'_i \cup \{ min(r_i)\}$ otherwise, has the desired properties.
\end{proof}

Before we state the following lemma we would like to make a comment:

\smallskip

Let $M \in [\mathbb{N}]^{\infty}$ and $n \in \N$. Partition $M$ in $n$-many infinite subsets in the following way
$$
M_j = \{m_i : i \equiv j \ \text{mod}(n) \},
$$
for each $1\leq j \leq n$. It is easy to see that if $\barr{B}$ on $M$ is a  spreading barrier  and  $s_j \in \barr{B}$ is the unique element such that $s_j \sqsubseteq M_j$, for each $1 \leq j \leq n$, then $(s_j)_{j=1}^{n} \in \plegman{\barr{B} \upharpoonright_{M}}{n}$.

\begin{lemma}
\label{plegmaaplegma}
Let $\barr{F}, \barr{G} \subseteq FIN$ be uniform barriers such that $\barr{G}$ is spreading and $o(\barr{F}) \leq o(\barr{G})$ and let  $M,N \in [\mathbb{N}]^{\infty}$. If $L_0 \in [N]^{\infty}$ satisfies the conclusion of Corollary \ref{2.5}, then for every $n \in \N$ and for every  $L \in [L_0]^{\infty}$ there exists $(s_i)_{i=1}^{n} \in \plegman{\barr{G} \upharpoonright L}{n}$ such that $(\psi_{L_0,M}(s_i))_{i=1}^{n} \in \plegman{\barr{F} \upharpoonright M}{n}$.
\end{lemma}

\begin{proof}
  First, we take $L' = \{l_i : i \equiv 0 \ (\text{mod n}) \}$ where $L = \{ l_i : i \in \N \}$. According to Lemma \ref{existenciaplegma}, we can take $(t_i)_{i=1}^{n} \in \plegman{T_{M, L_0}(\barr{F}\upharpoonright M)\upharpoonright L'}{n}$. By applying Lemma \ref{plegmagrande} to $(t_i)_{i=1}^{n}$ we obtain a plegma family  $(t'_i)_{i=1}^{n}$ so that $|t'_1|=|t'_2|= \cdots =|t'_n|$ and $t'_i \in [L]^{< \infty}$, for each $1 \leq i \leq n$. Now,  put $L'' = L / t'_n = \{l_i : i \geq max(t'_n)\}$ and, for each $1\leq i \leq n$, define the set
  $$
  L''_i = \{ l_j : j - max(t'_n) \equiv i \ \text{(mod n)} \}.
  $$
    As $\barr{G}$ is a barrier we can choose $s_i \in \barr{G} \upharpoonright L$ such that $s_i \sqsubseteq t'_i \cup L''_i$. Since  $\barr{G}$ is spreading it follows that $(s_i)_{i=1}^{n} \in \plegman{\barr{G} \upharpoonright L}{n}$. We assert that $\psi_{L_0,M}(s_i) = T^{-1}_{M,L_0}(t_i)$ for each $1 \leq i \leq n$. Indeed, fix $1 \leq i \leq n$. To establish this assertion it is enough, by definition of $\psi_{L_0,M}$, to prove that $T_{M,L_0}(T^{-1}_{M,L_0}(t_i)) = t_i \sqsubseteq s_i$. Notice from the construction that  $s_i \sqsubseteq t'_i \cup L''_i$ and $t_i \sqsubseteq t'_i$. Hence, it follows that either $s_i \sqsubset t_i$ or $t_i \sqsubseteq s_i$. The first relation never holds since, by hypothesis, we can always find $s'_i \in \barr{G}\upharpoonright L_0 $ such that $t_i \sqsubseteq s'_i$, but this would imply $s_i \sqsubset s'_i$ which is impossible because of $\barr{G}$ is a barrier.
\end{proof}

\begin{theorem}
Let $X$ be a Banach space, let $\barr{F}$ be a uniform barrier and let $\matb{x}{\barr{F}\upharpoonright_M}$ be a $\barr{F} \times \omega$-submatrix that generates $(e_i)_{i=1}^{\infty}$ as an asymptotic model. Then there exist a spreading barrier $\barr{B}$ with the same uniformity as $\mathcal{F}$, a $\barr{B} \times \omega$-matrix $\matb{y}{\barr{B}}$ and $N \in [\N]^{\infty}$ such that $\matb{y}{\barr{B} \upharpoonright_{N}}$ generates $(e_i)_{i=1}^{\infty}$ as an asymptotic model .
\end{theorem}

\begin{proof}
It is well known that we can find a  spreading uniform barrier $\barr{B}$ with the same uniformity as $\barr{F}$ and let $L_0 \in [\N]^{\infty}$ be as given by Theorem \ref{2.4}. Consider the function  $\psi_{L_0, M} : \barr{B} \upharpoonright_{L_0} \rightarrow \barr{F}$ given by  Corollary \ref{2.5}. By using this function, we define the $\barr{B} \times \omega$-submatrix so that $y^{i}_{s} =x^{i}_{\psi_{L_0, M}(s)}$ for each $s \in \barr{B} \upharpoonright_{L_0}$ and for each $i \in \N$.
To find $N$ first we fix an arbitrary $N_0 \in [L_0]^{\infty}$ and, by using Lemma \ref{plegmaaplegma} recursively, it is possible to choose, for each $i \in \N$, a subset $N_{i + 1} \in [N_i]^{\infty}$ so that $\psi_{L_0, M} \upharpoonright_{N_{i+1}}$ sends an element $(s_j)_{j =1}^{i+1} \in \plegman{\barr{B}\upharpoonright_{N_{i+1}}}{i+1}$ to $(\psi_{L_0, M}(s_j))_{j =1}^{i+1} \in \plegman{\barr{F} \upharpoonright_M}{i+1}$.
 Now let $N$ be a pseudointersection of the $N_i$'s.
Notice that, by definition, for every $(s_i)_{i=1}^{n} \in \plegman{\barr{B} \upharpoonright_N}{n}$ we have
$$
\norm{\sum_{i=1}^{n}a_i x^{i}_{\psi_{L_0, M}(s_i)}} = \norm{\sum_{i=1}^{n} a_i y^{i}_{s_i}},
$$
for each $(a_i)_{i=1}^{n} \in [-1,1]^{n}$. From this assertion it is  easy to see that the submatrix $(y^{i}_{s})_{s \in \barr{B} \upharpoonright_N, i \in \N}$ generates $(e_i)_{i=1}^{\infty}$  as an asymptotic model.
\end{proof}

Our next task is to  prove that if $\barr{B}$ is a uniform barrier, then every $\barr{B} \times \N$-matrix has a submatrix that, in some way, converges to a seminorm on $c_{00}$. In order to do this we need to introduce a function:

 \begin{definition}  Let $X$ be a Banach space and let $\barr{B}$ be a  barrier. For a  $\barr{B} \times \N$-matrix $\matb{x}{B}$, we define
  \begin{align*}
 \asyf{n} : \plegman{\barr{B}}{n} &\rightarrow \EspN{n}\\
 s &\rightarrow \asyf{n}(s)
 \end{align*}
 where  $\asyf{n}(s)(\summa{n}{a}{e}) = \norma{\summaa{n}{a}{x}{s}}$ for all $a_1,\cdots, a_n \in \mathbb{R}$.
 \end{definition}

That this seminorm $\asyf{n}(s)$ is indeed an element of $\EspN{n}$ for each $s  \in \plegman{\barr{B}}{n}$ follows from the normalized condition in the entries of the matrix.
  Observe that  the function $\asyf{n}$ depends on the given matrix, but since it will always be clear from context to which matrix we are referring to. Thus,  this function will  be always written  by   $\asyf{n}$ without mentioning the matrix.

\medskip

Before we state the next lemma, we recall that two seminorms $\rho_n \in \mathcal{N}_n$ and $\rho_m \in \mathcal{N}_m$, with $m>n$ are called compatible if $\rho_m \upharpoonright_{\mathbb{R}^{n}} = \rho_n $.

  \begin{theorem} \label{exisasym} Let $X$ be a Banach space, $\barr{B}$  a  barrier and   $\matb{x}{\barr{B}}$  a  $\barr{B}\times \N$-matrix in $X$. Then for every $N \in [\N]^{\infty}$ there exist  $M \in [N]^{\infty}$ and $\rho_n \in \EspN{n}$, for each  $n \in \mathbb{N}$, such that  the seminorms $\{ \rho_n: n \in \N \}$ are pairwise compatible and the function $\asyf{n}: \plegman{\barr{B} \upharpoonright_M}{n} \rightarrow \EspN{n} $ converges to  $\rho_n$, for every $n \in \mathbb{N}$.
 \end{theorem}

 \begin{proof} The set $M$ is going to be constructed recursively by applying Theorem~\ref{ramseyanalistas2}. Indeed, for $n=1$ we obtain $M_{1} \in [N]^\infty$ and a seminorm $\rho_1$ such that the function $\asyf{1}$ converges to  $\rho_1$ on $M_1$. Now, for every  $1 < n \in \N $,  we obtain $M_{n} \in [M_{n-1}]^\infty$ and a seminorm $\rho_n$ such that the function $\asyf{n}$ converges to  $\rho_n$ on $M_n$. Now, recursively, take $m_1 = \min(M_1)$ and $m_{n+1} = \min(M_{n+1} / m_n)$ for each $n \in \N$. Then, we define $M = \{m_n : n \in \mathbb{N} \}$.  It follows from the remark right after Theorem~\ref{ramseyanalistas} that $\asyf{n}$ converges to $\rho_n$ on $M$ for all $n \in \N$. To see that the seminorms $\{ \rho_n: n \in \N \}$ are indeed pairwise compatible notice that if  $n<m$, then $(s_i)_{i = 1}^{n} \in \plegman{\barr{B} \upharpoonright_ M}{n}$ for every $(s_i)_{i =1}^{m} \in \plegman{\barr{B} \upharpoonright_ M}{m}$. Hence, it follows that
 \begin{align*}
 d_{n}(\rho_n, \rho_m\upharpoonright_{\mathbb{R}^{n}}) &\leq d_n(\rho_n,\Psi^{n}((s_i)_{i =1}^{n})) + d_n(\Psi^{m}((s_i)_{i=1}^{m})) \upharpoonright_{\mathbb{R}^{n}}, \rho_m \upharpoonright_{\mathbb{R}^{n}}) \\
 &\leq d_n(\rho_n, \Psi^{n}((s_i)_{i=1}^{n})) + d_m(\Psi^{m}((s_i)_{i=1}^{ m})), \rho_m),
 \end{align*}
 for each $(s_i)_{i =1}^{m} \in \plegman{\barr{B} \upharpoonright_ M}{m}$. The conclusion follows from the fact that the functions $\Psi^{n}$ and $\Psi^{m}$ converge to $\rho_n$ and $\rho_m$, respectively, on $M$ and from Lemma~\ref{existenciaplegma}.
 \end{proof}

According to Theorem \ref{exisasym}, if $\matbr{x}{\barr{B}}$ is an $\barr{B}\times \N$-matrix on a Banach space $X$, it then follows from  Theorem \ref{exisasym} that there exists a submatrix $\matbr{x}{\barr{B}\upharpoonright_M}$ and a sequence $\suc{\rho}$ of compatible seminorms such that all the $\asyf{n}$'s are converging to  $\rho_n$ on $M$. This allows us to define a  seminorm $\rho = \bigcup_{n \in \mathbb{N}} \rho_n$ on the vector space $c_{00}(\N)$. In the case that $\rho$ is a norm on $c_{00}(\N)$, this normed space is the {\it asymptotic model generated} by the submatrix $\matbr{x}{\barr{B} \upharpoonright_M }$. Evidently, Theorem \ref{exisasym} extends the original result of A. Brunel and L. Sucheston \cite{bs} concerning spreading models.

\medskip

As we pointed after Definition 2.2,  given a uniform barrier $\barr{B}$ and $M \in [\N]^{\infty}$, the function $T_{M,\N} : \barr{B} \rightarrow FIN$ preserves plegma sequences and sends $\barr{B}$ to an uniform barrier of the same uniformity. Thus, it is  possible to obtain $\xi$-asymptotic models by using matrices instead of using submatrices.
  Notice that every basic sequence of $X$ is an element of $\MA{\xi}{X}$ for every non-zero ordinal number $\xi < \omega_1$; thus, most of the Banach spaces with Schauder basis admit more asymptotic models than spreading models. Now, we denote by $\SP{\xi}{X}$ the set of all $\xi$-asymptotic models generated by $\barr{B} \times \N$-matrices with the property that $x^{n}_{s} = x^{m}_s$ for each $s \in \barr{B}$ and $n,m \in \mathbb{N}$. It is easy to see that $\SP{\xi}{X}$ is exactly the set of all $\xi$-spreading models as were defined in the paper \cite{akt}.

\medskip

In the next theorem, we shall see that both spreading and asymptotic models are closely related one to the other.

\begin{lemma}
\label{subsucesion} Let $X$ be a Banach space and let  $0 < \xi < \omega_1$.
If $\suc{f} \in \MA{\xi}{X}$, then $(f_i)_{i \in M} \in \MA{\xi}{X}$ for every $M \in [\N]^{\infty}$.
\end{lemma}

\begin{proof}
Take a submatrix   $\matbr{x}{\barr{B}\upharpoonright_{N}}$ of a $\barr{B}\times \N$-matrix on $X$  that generates $\suc{f}$ as a $\xi$-asymptotic model. Now consider  the $\barr{B}\times \N$-matrix $(y^i_s)_{s \in \barr{B}, i \in \N}$ with entries $y_{s}^{i} = x_{s}^{m_i}$ if $s \in \barr{B}$ and $i \in \N$ , where $\{m_i : i \in \mathbb{N} \}$ is the increasing numeration of $M$. It is not hard to see that the submatrix $\matbr{y}{\barr{B}\upharpoonright_{N}}$ generates $(f_i)_{i \in M}$ as an asymptotic model.
\end{proof}

  \begin{theorem}
 \label{primero} For a non-zero ordinal $\xi$ we have the following properties:
 \begin{enumerate}
 \item $\SP{\xi}{X} \subseteq \MA{\xi}{X}$.

 \item If $\suc{e}$ is a spreading model generated by a subsequence of an element of $\MA{\xi}{X}$, then $\suc{e} \in \SP{\xi +1}{X}$.
 \end{enumerate}
 \end{theorem}

 \begin{proof}  $(1)$.  This follows from the definition of $\xi$-spreading model.

$(2)$. In virtue of Lemma~\ref{subsucesion}, we may begin by taking $\suc{f} \in \MA{\xi}{X}$ that generates $\suc{e}$ as a spreading model without going to a subsequence. Choose a  $\barr{B}\times \N$-submatrix $\matbr{x}{\barr{B}\upharpoonright_{N}}$ that generates $\suc{f}$ as a $\xi$-asymptotic model where $\mathcal{B}$ is a $\xi$-uniform barrier and $N \in [\N]^\infty$. Define the $(\mathbb{N} \oplus \barr{B}) \times \N$-matrix such that $x_{l ^\frown s}^{i} = x^{l}_{s}$ for $i \in \N$ and $l ^\frown s \in \mathbb{N} \oplus \barr{B}$. According to  Theorem~\ref{exisasym},
there is $M \in [N]^{\infty}$ such that $\asyf{n}$ converges on $M$, for every $n \in \N$, and denote the corresponding spreading model by $(h_i)_{i \in \N}$. We will show that this  $(\xi +1)$-spreading model is actually $\suc{e}$. In fact, this assertion follows from the next claim:

\medskip

{\bf Claim:}
For every $k \in \N$,  $m \in M$ and  $\epsilon > 0$ there exists a plegma sequence $(n_i^{\frown} s_i)_{i=1}^{k} \in \plegman{\barr{B} \upharpoonright_{M / m} }{k}$  such that
  $$
  \left|\norm{\summa{k}{a}{e}}_E - \norm{{\sum_{i = 1}^{k} a_i x^{i}_{n_i ^\frown s_i} }}_X\right| < \epsilon,
  $$
for every $(a_i)_{i=1}^{k} \in [-1,1]^{k}$.

\medskip

{\it Proof of the Claim:} Fix $k \in \mathbb{N}$, $m \in M$ and $\epsilon > 0$. Choose $\{n_1, n_2,...,n_k\} \in [M / m]^{k}$ such that
  $$
  \left| \norm{\sum_{i=1}^{k} a_i e_i}_E - \norm{\sum_{i=1}^{k} a_i f_{n_i}}_F \right| < \frac{\epsilon}{2},
  $$
  for every $(a_i)_{i=1}^{k} \in [-1,1]^{k}$. Now, by using asymptotic convergence and Lemma~\ref{existenciaplegma}, we may take a plegma sequence $(s_i)_{i=1}^{n_k} \in \plegman{\barr{B}\upharpoonright_{M/ n_k}}{n_k}$ such that
  $$
  \left| \norm{\sum_{j=1}^{n_k} b_j f_{j}}_F - \norm{\sum_{j=1}^{n_k} b_j x^{j}_{s_j}}_X \right| < \frac{\epsilon}{2}
  $$
 for each $(b_j)_{j=1}^{n_k} \in [-1,1]^{n_k}$. Fix  $(a_i)_{i=1}^{k} \in [-1,1]^{k}$ and consider the  sequence $b_{n_i} = a_i$, for each $1 \leq i \leq k$, and $b_j = 0$ for $j \notin \{n_1,...,n_k\}$. Then,  $(n_i ^\frown s_{n_i})_{i=1}^{k} \in \plegman{(\N \oplus \barr{B}) \upharpoonright_{M}}{k}$ satisfies  that

 \begin{align*}
 &\left|\norm{\summa{k}{a}{e}}_E - \norm{{\sum_{i = 1}^{k} a_i x^{i}_{n_i ^\frown s_i} }}_X\right|  \\
 &= \left|\norm{\summa{k}{a}{e}}_E - \norm{\summaaa{k}{a}{x}{s}{n}}_X\right|  \\
  &\leq \left|\norm{\summa{k}{a}{e}}_E -\norm{\summasa{k}{a}{f}{n}}_F  \right| + \left| \norm{\summasa{k}{a}{f}{n}}_F - \norm{\summaaa{k}{a}{x}{s}{n}}_X\right| \\
  &\leq \left|\norm{\summa{k}{a}{e}}_E -\norm{\summasa{k}{a}{f}{n}}_F  \right| + \left| \norm{\sum_{j=1}^{n_k} b_j f_{j}}_F - \norm{\sum_{j=1}^{n_k} b_j x^{j}_{s_j}}_X \right| \\
  &\leq \frac{\epsilon}{2} + \frac{\epsilon}{2} = \epsilon.
 \end{align*}

\smallskip

 The result follows from the inequality

$$
\left|\norm{\summa{k}{a}{e}}_E  - \norm{\summa{k}{a}{h}}_H \right| \leq
 $$
 $$
 \left|\norm{\summa{k}{a}{e}}_E - \norm{{\sum_{i = 1}^{k} a_i x^{i}_{n_i ^\frown s_i} }}_X  \right| - \left| \norm{{\sum_{i = 1}^{k} a_i x^{i}_{n_i ^\frown s_i} }}_X - \norm{\summa{k}{a}{h}}_H \right|
$$

and by using the claim and the convergence of $(x_{l ^\frown s}^{i})_{l ^\frown s \in (\N \oplus \barr{B})/M , i \in N}$ to choose an appropriate plegma sequence $(n_i^{\frown} s_i)_{i=1}^{k} \in \plegman{\barr{B} \upharpoonright_{M}}{k}$.
 \end{proof}

Now, we generalize the notion of a block sequence, in the context of $\barr{B} \times \N$-matrices, following ideas from \cite{akt2}.

\begin{definition}
Let $X$ be a Banach space with Schauder basis $\suc{e}$. A  $\barr{B} \times \N$-matrix $\matbr{x}{\barr{B}}$ on $X$ is called a {\it plegma block} $\barr{B} \times \N$-{\it matrix} if all its entries are block vectors of $\suc{e}$ and for every $n \in \N$ there exists an $m \in \N$ such that for each $(s_i)_{i=1}^{n} \in \plegman{\barr{B} / m}{n}$ the sequence $(x^{i}_{s_i})_{i=1}^{n}$ is a block subsequence of $\suc{e}$.
\end{definition}

The asymptotic model version of Theorem $42$ from \cite{akt2} is stated in the following theorem. To prove it we shall need the next two lemmas.

  \begin{definition}
  Let $\barr{B}$ be a uniform barrier, $n \in \N$ and $s \in \barr{B}$ such that $n < s$. A sequence $(r_i)_{i=1}^{n} \in [[\N]^{< \infty}]^{n}$ is called a $\plegman{\barr{B}}{n}$-{\it decomposition} of $s$ if  $r_i$ is the unique element of $\barr{B}$ satisfying $r_i \sqsubseteq s - (n - i)$, for every $1 \leq i \leq n$, and $(r_i)_{i=1}^{n} \in \plegman{\barr{B}}{n}$.
  \end{definition}

In the next, lemma we shall see that we can find plegma families whose elements have suitable decompositions with very strong combinatorial properties.

  \begin{lemma}
  \label{plegmasuma}
  Let $\barr{B}$ be a uniform spreading barrier and $l, n, m_1, m_2, \cdots, m_n$  elements of $\N$. Then for every $M \in [\N]^{\infty}$
  there exists $N \in [M]^{\infty}$ such that each $(s_i)_{i=1}^{n} \in \plegman{\barr{B}\upharpoonright_{N}}{n}$ satisfies that  $s_i$ has a $\plegman{\barr{B}}{m_i}$-decomposition $(r^{i}_j)_{j=1}^{m_i}$, for every $1 \leq i \leq n$,  and
  $$
  l < (r^{1}_j)_{j=1}^{m_1} {}^\frown  \cdots {}^\frown(r^{n}_j)_{j=1}^{m_n} \in \plegman{\barr{B}}{\sum_{i=1}^{n} m_i}.
  $$
  \end{lemma}
  \begin{proof}
  Take $k > max\{m_1, m_2, \cdots, m_n \} $ and
  $$
  N = \{m_i : i \equiv 0 \ mod(k) \}/ \{l+k \}.
  $$
  It is not hard to see that $N$ satisfies de conclusion.
  \end{proof}

  \begin{theorem}
 \label{segundo} Let $k \in \N \setminus \{0\}$ and $\xi< \omega_1$.
 If $\suc{e}$ is a $k$-asymptotic model generated by a plegma block $\mathbb{N}^{[k]} \times \N$-submatrix of the space generated by $\suc{f} \in \MA{\xi}{X}$, then $\suc{e} \in \MA{\xi + k}{X}$.
 \end{theorem}
 \begin{proof} Let $\matb{x}{\barr{B}}$ be a plegma block matrix generating $\suc{f}$ as an $\xi$-asymptotic model and $(y_t^{j} )_{t \in [\N]^{k}, j \in \N}$ a plegma block $\mathbb{[N]}^{k} \times \N$-matrix of $\suc{f}$ generating $\suc{e}$ as a $k$-asymptotic model, since it is a plegma block matrix each entry can be expressed as $y_t^{j} = \summaset{F_{(t,j)}}{a^{(t,j)}}{f}$, where $a^{(t,j)}_i \in \mathbb{R}$ for each $t \in [\N]^{k} $ and for each $j \in \N$. Now, we shall define a $ ([\mathbb{N}]^{k} \oplus \barr{B}) \times \N$-matrix on $X$ which will contain a submatrix generating $\suc{e}$ as an asymptotic model. To do this, we first take for every $t ^\frown s \in \mathbb{[N]}^{k} \oplus \barr{B}$
$$
 z^{j}_{t ^\frown s} = \frac{\summaaset{F_{(t,j)}}{a^{(t,j)}}{x}{r}}{\norm{\summaaset{F_{(t,j)}}{a^{(t,j)}}{x}{r}}} \   \text{if} \hspace{ .2cm} \max{F_{(t,j)}} \leq s \hspace{.2 cm}
 $$
 and $\hspace{.2cm}  (r_i)_{i \leq \max{F_{(t,j)}}} \in \plegman{\barr{B}}{\max{F_{(t,j)}}}$.
 $$
 z^{j}_{t ^\frown s}  =  \text{any fixed element of} \ S(X) \hspace{ .2cm} \text{otherwise},
 $$
 where $(r_i)_{i \leq max{F_{(t,j)}}}$ is the $\plegman{\barr{B}}{\max{F_{(t,j)}}}$-decomposition of $s$. By Theorem~\ref{exisasym} there exists an infinite $M \in [\N]^{\infty}$ such that the matrix restricted to $M$ converges to some seminorm. We claim that this seminormed space is actually the Banach space  $\suc{e}$. Hence, $(z^{j}_{t ^\frown s})_{t ^\frown s \in \barr{B}\upharpoonright M, j \in \N}$  generates $\suc{e}$ as a $(\xi + k)$-asymptotic model. To start  proving our claim fix $\epsilon >0 $ and $n \in \N$, and choose $(t_i)_{i=1}^{n} \in \plegman{[\N]^{k}}{n}$ so that
 $$
 \left|\norm{\summa{n}{b}{e}} - \norm{\summas{n}{b}{y^{i}}{t}}\right| < \frac{\epsilon}{2}
 $$
 for each $(b_i)_{i=1}^{n} \in [-1,1]^{n}$. Now, we apply Lemma \ref{plegmasuma} with $n$, $\barr{B}$, $M \in [\N]^{\infty}$ and the natural numbers $\max{F_{(t_1,1)}}, \max{F_{(t_2,2)}}, \cdots,$ $ \max{F_{(t_n,n)}}$ to obtain $N \in [\N]$ such that every $(s_i)_{i=1}^{n} \in \plegman{\barr{B} \upharpoonright_{N}}{n}$ satisfies:
\begin{enumerate}
 \item  The $\plegman{\barr{B}}{\max{F_{(t_i,i)}}}$-decomposition of $s_i$, $$(r_{(i,v)})_{ v \leq \max{F_{(t_i, i)}}} \in \plegman{\barr{B}}{\max{F_{(t_i,i)}}},$$ for every $i \leq n$.

 \item $(r_{(0,v)})_{ v \leq \max{F_{(t_0, 0)}}} ^\frown (r_{(1,v)})_{ v \leq \max{F_{(t_1, 1)}}} ^\frown \cdots ^\frown (r_{(n,v)})_{ v \leq \max{F_{(t_n, n)}}}$ is an element of $\plegman{\barr{B}}{\sum_{i \leq      n}\max{F_{(t_i,i)}}}$.
 \item $l < r_{(0,v)}$ for $l \in \N$.

\end{enumerate}
 Notice also that $(t_i ^{\frown} s_i)_{i=1}^{n} \in \plegman{\mathbb{[N]}^{k} \oplus \barr{B}}{n}$. Then, by taking a sufficiently large $l \in \N$, we obtain that
$$
 \left|\norm{\summa{n}{b}{e}} - \norm{\sum_{i=1}^{n} b_i z^i_{t_i ^\frown s_i} } \right| \leq
 $$
 $$
 \left|\norm{\summa{n}{b}{e}} - \norm{\summas{n}{b}{y^{i}}{t}}\right|  + \left|\norm{\summas{n}{b}{y^{i}}{t}}- \norm{\sum_{i =1}^{n} b_i z^i_{t_i ^\frown s_i} } \right|
 $$
 $$
 \leq \frac{\epsilon}{2} + \left|\norm{\sum_{i =1}^{n} b_i \sum_{j \in F_{(t_i,i)}} a_j^{(t_i,i)} f_j} - \norm{\sum_{i =1}^{n} b_i \frac{\sum_{j \in F_{(t_i,i)}} a_j^{(t_i,i)} x^{j}_{r_{(i,j)}}}{\norm{\sum_{j \in F_{(t_i,i)}} a_j^{(t_i,i)} x^{j}_{r_{(i,j)}}}}}\right|
$$
for all $(b_i)_{i=1}^{n} \in [-1,1]^{n}$. We can make the second summand smaller than $\frac{\epsilon}{2}$. Indeed, by taking $(s_i)_{i=1}^{n} \in \plegman{\barr{B}}{n}$ with $l$ large enough we can get
 $$
 \left| \norm{\sum_{j \in F_{(t_i,i)}} a_j^{(t_i,i)} x^{j}_{r_{(i,j)}}} - 1 \right| = \left| \norm{\sum_{j \in F_{(t_i,i)}} a_j^{(t_i,i)} x^{j}_{r_{(i,j)}}} - \norm{\sum_{j \in F_{(t_i,i)}} a_j^{(t_i,i)} f_j }\right| \leq \frac{\epsilon}{4n}
 $$
 for each $1 \leq i \leq n$.
 Therefore,
\begin{align*}
&\left|\norm{\sum_{i =1}^{ n} b_i \sum_{j \in F_{(t_i,i)}} a_j^{(t_i,i)} f_j} - \norm{\sum_{i =1}^{ n} b_i \frac{\sum_{j \in F_{(t_i,i)}} a_j^{(t_i,i)} x^{j}_{r_{(i,j)}}}{\norm{\sum_{j \in F_{(t_i,i)}} a_j^{(t_i,i)} x^{j}_{r_{(i,j)}}}}}\right| \\
 &\leq \left|\norm{\sum_{i =1}^{ n} b_i \sum_{j \in F_{(t_i,i)}} a_j^{(t_i,i)} f_j} -  \norm{\sum_{i =1}^{ n} b_i \sum_{j \in F_{(t_i,i)}} a_j^{(t_i,i)} x^{j}_{r_{(i,j)}}}\right| \\
 &+ \left| \norm{\sum_{i =1}^{ n} b_i \sum_{j \in F_{(t_i,i)}} a_j^{(t_i,i)} x^{j}_{r_{(i,j)}}}  -   \norm{\sum_{i =1}^{ n} b_i \frac{\sum_{j \in F_{(t_i,i)}} a_j^{(t_i,i)} x^{j}_{r_{(i,j)}}}{\norm{\sum_{j \in F_{(t_i,i)}} a_j^{(t_i,i)} x^{j}_{r_{(i,j)}}}}} \right|
 \end{align*}

 Since $(y_t^{j} )_{t \in [\N]^{k}, j \in \N}$ is a plegma block matrix, it follows that
 $$F_{(t_0, 0)} < F_{(t_1, 1)} < \cdots < F_{(t_n, n)}.$$ So, by "cutting" some of the elements in $(2)$, we consider $r$ defined as
  $$(r_{(0,v)})_{ v \leq \max{F_{(t_0, 0)}}} ^\frown (r_{(1,v)})_{\max{F_{(t_0, 0)}} v \leq \max{F_{(t_1, 1)}} } ^\frown \cdots ^\frown (r_{(n,v)})_{\max{F_{(t_{n-1}, n-1)}} v \leq \max{F_{(t_n, n)}}} $$
  which is an element of $\plegman{\barr{B}}{\max{F_{(t_n,n)}}}$. Then we conclude that the first summand of the previous inequality is bounded by
 $$
 d_{\max{F_{(t_n, n)}}}(<f_i>_{i=1}^{\max{F_{(t_n, n)}}}, \Phi(r))
 $$
 , where $<f_i>_{i=1}^{\max{F_{(t_n, n)}}}$ denotes the element of $\mathcal{N}_{\max{F_{(t_n, n)}}}$ that inherits its norm, and this can be made smaller than $\frac{\epsilon}{4}$ by applying the hypothesis of convergence from $\matb{x}{\barr{B}}$. By using this assertion, our previous inequality becomes:
 \begin{align*}
 &< \frac{\epsilon}{4} + \norm{\sum_{i =1}^{ n} b_i \sum_{j \in F_{(t_i,i)}} a_j^{(t_i,i)} x^{j}_{r_{(i,j)}} -\sum_{i =1}^{ n} b_i \frac{\sum_{j \in F_{(t_i,i)}} a_j^{(t_i,i)} x^{j}_{r_{(i,j)}}}{\norm{\sum_{j \in F_{(t_i,i)}} a_j^{(t_i,i)} x^{j}_{r_{(i,j)}}}}} \\
 &< \frac{\epsilon}{4} +  \norm{\sum_{i =1}^{ n} b_i \sum_{j \in F_{(t_i,i)}} a_j^{(t_i,i)} x^{j}_{r_{(i,j)}} [1  - \frac{ 1}{\norm{\sum_{j \in F_{(t_i,i)}} a_j^{(t_i,i)} x^{j}_{r_{(i,j)}}}}]} \\
 &< \frac{\epsilon}{4} + n  \left| \norm{\sum_{j \in F_{(t_i,i)}} a_j^{(t_i,i)} x^{j}_{r_{(i,j)}}} -1 \right| \leq \frac{\epsilon}{4} + n\frac{\epsilon}{4n} = \frac{\epsilon}{2}.
\end{align*}
For each $(b_i)_{i=1}^{n} \in [-1,1]^{n}$.
 \end{proof}

It is shown in \cite[Th. 68]{akt2} that there exists a space $X$ without $\ell_p$ as a spreading model of any level. So, we can conclude from the Theorem~\ref{primero}  that $X$ does not have $\ell_p$ as an asymptotic model of any order. Moreover,  Theorem~\ref{segundo} implies that $X$ does not admit a finite asymptotic block chain $(e^{i+1}_j)_{j \in \mathbb{N}} \in
 \MA{}{\overline{(e^{i}_j)_{j \in \mathbb{N}}}}$, for every $i \leq n$, such that   $(e^{n}_j)_{j \in \mathbb{N}} \cong \ell_p$. Thus
Question \ref{6.5} has a negative answer.

\medskip

It is well know that, in a Banach space with Schauder basis, every spreading model generated by a weakly null sequence can also be generated by a block sequence. This situation may be extended to the asymptotic case as it is shown in the next theorem, but to have  this done we first recall the next definition from \cite{asym}:


\begin{definition}  A matrix $(x^{n}_i)_{i \in \N, n \in \N}$ is called a {\it weakly null matrix} if:
$$
(x^{n}_{i})_{i \in \N}
$$
is a weakly null sequence, for every $n \in \N$.
\end{definition}


\begin{theorem}\label{ultimo}
Let $(x^{n}_i)_{i \in \N, n \in \N}$ be a weakly null matrix generating $(f_i)_{i \in \N}$ as an asymptotic model. Then there exists $(y^n_i)_{i \in \N, n \in N}$ a plegma block matrix that generates $(f_i)_{i \in \N}$ as an asymptotic model.
\end{theorem}

\begin{proof}
  By a simple use of the ``gliding hump'' argument, we can recursively choose block subsequences $(y^n_i)_{i \in \N}$ and $(k^{n+1}_i)_{i \in \N}$ (each sequence is a subsequence of the previous one) so that
  $$
  \norm{y^{n+1}_i - x^{n+1}_{k^{n+1}_i}} \leq \frac{1}{2^{(n+1)+i+1}},
   $$
for each $n \in  \mathbb{N}$.   Let us see that $(y^n_i)_{i \in \N, n \in \N}$ certainly generates $(f_i)_{i \in \N}$ as an asymptotic model. Indeed, this will follow from the next inequality:
  \begin{align*}
  &\left| \norm{\sum_{i =1}^{l} a_i f_i} - \norm{\sum_{i =1}^{l} a_i y^i_{s(i)}} \right| \\
  &\leq \left| \norm{\sum_{i =1}^{l} a_i f_i} - \norm{\sum_{i =1}^{l} a_i x^i_{k_{s(i)}}} \right|  + \left| \norm{\sum_{i =1}^{l} a_i x^i_{k_{s(i)}}}  - \norm{\sum_{i =1}^{l} a_i y^i_{s(i)}} \right| \\
  &\leq \left| \norm{\sum_{i =1}^{l} a_i f_i} - \norm{\sum_{i =1}^{l} a_i x^i_{k_{s(i)}}} \right| + \frac{1}{2^{s(1)}},
  \end{align*}
for a given finite subset $s = \{s(1), s(2), \cdots , s(l) \}$ of natural numbers  and for every $(a_i)_{i=1}^l \in [-1,1]$. Notice that the first summand can be made arbitrarily small by using the assumption on the matrix $(x^{n}_i)_{i \in \N, n \in \N}$. Thus, $(y^n_i)_{i \in \N, n \in \N}$ generates $(f_i)_{i \in \N}$ as an asymptotic model.
\end{proof}

 It is easy to see from Theorems $\ref{segundo}$ and $\ref{ultimo}$ that any finite chain of asymptotic models generated by weakly null matrices (inside their respective Banach spaces) can be replaced by a chain consisting of block matrices. This remark  leads us to the next corollary.

\begin{corollary}
There is a Banach space $X$ such that no finite chain of weakly generated asymptotic models $X_1, X_2, \cdots, X_n$ starting with $X$ has a an ending space either $l_p$ for $p \in [1, \infty)$ or $c_0$.
\end{corollary}

 \medskip

We finish the paper by listing some questions

 \medskip

As in the spreading case we also have that $AM_{\xi}(X) \subseteq AM_{\eta}(X)$ provided that $\xi < \eta  < \omega_1$.

\begin{question}
For distinct $\xi < \eta  < \omega_1$  does there exist a Banach space $X$ such that
$$
AM_{\xi}(X) \neq AM_{\eta}(X) \ ?
$$
\end{question}


 For a $\mathcal{F} \subseteq FIN$ and $t \in FIN$ we denote by $\mathcal{F}_t = \{ s \in FIN : t < s \ \text{and} \ t^{\frown} s \in \mathcal{F}\}$. If $B \subseteq FIN$ is a uniform barrier, then it is known that for every $t \in FIN$ such that $\mathcal{B}_{t}$ has uniformity greater than $1$ it is possible to find $s \in FIN$ such that $\mathcal{B}_{t^{\frown}s}$ has uniformity exactly one.

 \medskip

In the paper \cite{aost} it was proven that for any countable set of spreading models has an upper bound in the pre-order of domination. This was generalized in \cite{st} by establishing this same property for $SM^{\xi}_w(X)$, where $SM^{\xi}_w(X)$ denotes the family of all $\xi$-order spreading models of a Banach space $X$ generated by subordinated weakly null $\mathcal{F}$-sequences. Furthermore, if this set contains an increasing sequence of length $\omega$, then it contains an increasing sequence of length $\omega_1$. In view of these facts, it is natural to consider
the following class of weakly null matrices.

\begin{definition} Let $\mathcal{B}$ be a uniform barrier. A $\mathcal{B} \times \N$-matrix $(x^{n}_{s})_{s \in \mathcal{B}, n \in \N}$ is called {\it weakly null} if for every $n \in \N$ and for each $s \in FIN$ such that $\barr{B}_{s}$ is of uniformity $1$
$$
(x^{n}_{s ^{\frown} i})_{i \in \N / s}
$$
is weakly null matrix.
\end{definition}

Then, we may ask if the set $AM^{\xi}_w(X)$ of asymptotic models generated by weakly null matrices is also a semi-lattice with the pre-partial order of domination.

\medskip

 Next, we pose the questions  that are the higher asymptotic version of Problems $6.2$, $6.3$ and $6.6$ from the article \cite{asym}, respectively.

\medskip

In the following questions,  fix $\xi < \eta  < \omega_1$ .

\medskip

\begin{question} Let  $\barr{B}$ be a barrier of uniformity $\xi$ and let  $X$ be a Banach space with a Schauder basis. If $X$ admits a unique, isometric, asymptotic model of level $\xi$ for all normalized block basic weakly null $\barr{B} \times \N$-matrices, must $X$ contain an isomorphic copy of that space \ ?
\end{question}

 It is known that  if $X$ admits a unique, isometric, asymptotic model of level $\xi$ for all normalized block basic weakly null $\barr{B} \times \N$-matrices this unique asymptotic model must be $c_0$ or $\ell_p$ for some $1 \leq p < \infty$.

\begin{question}
Is it possible to stabilize the $\xi$-asymptotic structure of the space $X$  \ ?
\end{question}

\begin{question}
Can the distortion of $\ell_p$ $(1 \leq p < \infty)$ be extended to its $\xi$-asymptotic structure ? That is, are there a $K>1$ and an equivalent norm $||| \cdot |||$ on $\ell_p$ such that none
$\xi$-asymptotic  model of this renorm is $K$-equivalent to the unit vector basis of $\ell_p$ \ ?
\end{question}



\def\polhk#1{\setbox0=\hbox{#1}{\ooalign{\hidewidth
  \lower1.5ex\hbox{`}\hidewidth\crcr\unhbox0}}}
  \def\polhk#1{\setbox0=\hbox{#1}{\ooalign{\hidewidth
  \lower1.5ex\hbox{`}\hidewidth\crcr\unhbox0}}} \def\cprime{$'$}
  \def\cprime{$'$}

\end{document}